\documentclass[11pt]{article}

\usepackage[T1]{fontenc}
\usepackage{lmodern}
\usepackage{microtype}
\usepackage{setspace}
\usepackage[margin=0.9in]{geometry}

\usepackage{amsmath,amssymb,amsfonts,amsthm}

\usepackage{cite}
\usepackage{xcolor}
\usepackage[
  pdfusetitle,
  colorlinks=true,
  linkcolor=blue,
  citecolor=blue,
  urlcolor=blue,
  pdfborder={0 0 0}
]{hyperref}
\usepackage{titling}
\usepackage{comment}
\usepackage{booktabs}
\usepackage{etoolbox}
\usepackage{fancyhdr}

\AtBeginEnvironment{thebibliography}{\small}


\pretitle{\begin{flushleft}\Large}
\posttitle{\par\end{flushleft}\vskip 0.5em}

\preauthor{\begin{flushleft}}
\postauthor{%
  \par\vspace{0.5em}\footnotesize
  California Institute of Technology\\
  Email: \href{mailto:looi@caltech.edu}{looi@caltech.edu}
  \par\end{flushleft}
}

\predate{\begin{flushleft}}
\postdate{\par\end{flushleft}}

\theoremstyle{plain}
\newtheorem{theorem}{Theorem}[section]
\newtheorem{proposition}[theorem]{Proposition}
\newtheorem{lemma}[theorem]{Lemma}

\theoremstyle{definition}

\theoremstyle{remark}
\newtheorem{remark}[theorem]{Remark}

\newcommand{\B}{\mathbb{B}}
\newcommand{\C}{\mathbb{C}}

\newcommand{\dv}{\,dv}
\newcommand{\dvgamma}{\,dv_\gamma}

\newcommand{\norm}[1]{\left\lVert #1\right\rVert}
\newcommand{\abs}[1]{\left|#1\right|}
\newcommand{\inner}[2]{\left\langle #1,#2\right\rangle}

\title{\bf When the Berezin transform fails to detect compactness: Toeplitz operators with $L^1$ symbols on weighted Bergman spaces}
\author{{\bf Sam Looi}}
\date{\small\today}

\begin{document}

\maketitle
\begin{abstract}
For every $n\geq1$ and $\gamma>-1$, we construct $f\in L^1(\mathbb B_n,dv_\gamma)$ whose Toeplitz form extends to a bounded, noncompact operator on the weighted Bergman spaces $A^2_\gamma(\mathbb B_n)$ and whose Berezin transform vanishes at the boundary. Boundary vanishing of the Berezin transform therefore does not imply compactness for Toeplitz operators with integrable symbols, in contrast with operators in the Toeplitz algebra generated by bounded symbols; this answers a question of Bauer and Isralowitz. The construction approximates rank-one operators in norm by Toeplitz operators with smooth, compactly supported symbols and transports suitably separated blocks toward the boundary. On the unweighted disk, a diagonal version produces a real symbol with the same vanishing and noncompactness properties, for which the operator is positive and Zorboska's two-sided localization condition holds at $p=3$. Her hypothesis $p>3$ is shown to be sharp.
\end{abstract}



\section{Introduction}

On Bergman spaces, the normalized reproducing kernels converge weakly to zero as their base points approach the boundary. If $T$ is a compact operator, then compactness forces the diagonal matrix coefficients against these kernels---the Berezin transform---to vanish there. The converse question asks for a scalar family of diagonal coefficients to detect a global operator property. This converse holds throughout the norm-closed algebra generated by Toeplitz operators with bounded symbols \cite{Suarez2007,MitkovskiSuarezWick2013}. It is therefore natural to ask whether boundary vanishing continues to characterize compactness for Toeplitz operators with less regular symbols. 

We show that it does not, even for globally integrable symbols. For every $n\geq1$ and $\gamma>-1$, we construct $f\in L^1(\mathbb B_n,dv_\gamma)$ whose Toeplitz form has a bounded extension $T_f$ on $A^2_\gamma(\mathbb B_n)$ and satisfies \[ B_\gamma(f)(z)\longrightarrow0 \qquad\text{as }|z|\to1, \]
although $T_f$ is not compact. This answers a question of Bauer and Isralowitz \cite[Section~7]{BauerIsralowitz2012}. A separate construction on the unweighted disk has the same vanishing and noncompactness properties, with $f$ real-valued, $T_f$ positive, and Zorboska's two-sided localization condition (recalled in \eqref{eq:zorboska-condition} below) satisfied at $p=3$. Thus the strict hypothesis $p>3$ in her compactness theorem is sharp \cite{Zorboska2003}.

Let $\B_n$ be the unit ball in $\C^n$, where $n\geq 1$, and fix $\gamma>-1$.  If $dv$ denotes Lebesgue measure on $\C^n$, we set
\[
dv_\gamma(z)
=c_\gamma(1-|z|^2)^\gamma\dv(z),
\qquad
c_\gamma=\frac{\Gamma(n+\gamma+1)}
{\pi^n\Gamma(\gamma+1)}.
\]
Thus $dv_\gamma$ is a probability measure.  The weighted Bergman space $A^2_\gamma=A^2_\gamma(\B_n)$ is the closed subspace of $L^2(\B_n,dv_\gamma)$ consisting of holomorphic functions. Let
\[
N=n+1+\gamma.
\]
With the convention that the inner product is linear in the first variable, the reproducing kernel and the normalized reproducing kernel are
\[
K_z(w)=\frac{1}{(1-\inner{w}{z})^N},
\qquad
k_z(w)=\frac{(1-|z|^2)^{N/2}}{(1-\inner{w}{z})^N}.
\]
For a bounded operator $A$ on $A^2_\gamma$, its Berezin transform is
\[
B_\gamma(A)(z)=\inner{Ak_z}{k_z}.
\]
For $f\in L^1(\B_n,dv_\gamma)$, we use the same notation for
\[
B_\gamma(f)(z)
=\int_{\B_n}f(w)|k_z(w)|^2\dvgamma(w).
\]

An $L^1$ symbol need not define a bounded Toeplitz operator by the usual projection formula, so we formulate the definition through a sesquilinear form.  For holomorphic polynomials $p$ and $q$, the Toeplitz form
\[
(p,q)\mapsto\int_{\B_n}f(z)p(z)\overline{q(z)}\dvgamma(z)
\]
is well-defined because polynomials are bounded on $\B_n$.  We say that the Toeplitz form with symbol $f$ has a bounded extension if there is a bounded operator $T_f$ on $A^2_\gamma$ such that
\[
\inner{T_fp}{q}=\int_{\B_n}f(z)p(z)\overline{q(z)}\dvgamma(z)
\]
for all holomorphic polynomials $p,q$.  Such an operator is unique.  In this case Lemma~\ref{lem:l1-covariance} gives
\[
B_\gamma(T_f)=B_\gamma(f),
\]
and also identifies $T_f$ with the usual pointwise kernel formula on bounded holomorphic inputs. For Toeplitz operators with $L^1$ symbols on the ball, see \cite{DieudonneTchoundja2011}; for locally integrable symbols on the disk, see \cite[Theorems~2.3 and~2.6]{TaskinenVirtanen2010}.

Let $\mathcal T_\gamma$ be the norm-closed algebra generated by the Toeplitz operators with bounded symbols.  On the unweighted Bergman space of the disk, Axler and Zheng first proved the relevant equivalence for finite sums of finite products of Toeplitz operators with bounded symbols \cite[Theorem~2.2]{AxlerZheng1998}.  Su\'arez proved it for the full Toeplitz algebra on the unweighted ball, and Mitkovski, Su\'arez, and Wick extended the result to the standard weighted spaces.  Thus, for $A\in\mathcal T_\gamma$,
\[
A\text{ is compact}
\quad\Longleftrightarrow\quad
B_\gamma(A)(z)\longrightarrow 0
\quad\text{as }|z|\to 1
\]
\cite[Theorem~9.5]{Suarez2007} \cite[Theorem~1.1]{MitkovskiSuarezWick2013}. Our main result shows that membership in $\mathcal T_\gamma$ cannot be omitted.

\begin{theorem}\label{thm:main}
For every $n\geq 1$ and $\gamma>-1$, there exist $f\in L^1(\B_n,dv_\gamma)$ and a sequence $(a_m)$ in $\B_n$ with $|a_m|\to 1$ such that the Toeplitz form with symbol $f$ has a bounded extension $T_f$,
\[
B_\gamma(f)(z)\longrightarrow 0  \qquad\text{as }|z|\to 1,
\]
but
\[
\inf_m\,\norm{T_fk_{a_m}}_{A^2_\gamma}>0.
\]
In particular, $T_f$ is not compact.
\end{theorem}

Indeed, $k_{a_m}\to0$ weakly, so the last estimate implies noncompactness.  The operator $T_f$ in Theorem~\ref{thm:main} necessarily lies outside $\mathcal T_\gamma$; otherwise the preceding compactness theorem would apply.  In particular, the constructed symbol is not essentially bounded.

\subsection*{Relation to the Toeplitz algebra and localization}
For every $(n-1)/(n+1)<s<1$, let $\mathcal A_s$ denote the class of $s$-weakly localized operators on $A^2_0(\B_n)$ introduced in \cite[Definition~1.1]{Xia2015}. Xia proved \[ \mathcal T_0 = \overline{\{T_g:g\in L^\infty(\B_n,dv_0)\}}^{\|\cdot\|} = C^*(\mathcal A_s) ;\]
see \cite[Theorems~1.3 and~1.5]{Xia2015}. Since $\mathcal A_s$ is a $*$-algebra, $C^*(\mathcal A_s)$ is its operator norm closure.
Dewage extended these identities to every weight considered here. If $\mathcal W_{\rm loc}^\alpha(A^2_\gamma)$ denotes Dewage's class of $\alpha$-weakly localized operators, then \[ \mathcal T_\gamma = \overline{\{T_g:g\in L^\infty(\B_n,dv_\gamma)\}}^{\|\cdot\|} = \overline{\mathcal W_{\rm loc}^\alpha(A^2_\gamma)}^{\|\cdot\|}, \qquad \frac{n-1-\gamma}{2}<\alpha<\frac{n+1+\gamma}{2} \] \cite[Theorem~5.9]{Dewage2026}. As a result, the operator in Theorem~\ref{thm:main} is not an operator norm limit of Toeplitz operators with bounded symbol and does not belong to $\overline{\mathcal W_{\rm loc}^\alpha(A^2_\gamma)}^{\|\cdot\|}$ for any admissible $\alpha$. In particular, it is not $\alpha$-weakly localized for any such $\alpha$. 

The construction below nevertheless realizes $T_f$ as the strong operator limit of Toeplitz operators with bounded symbol: $f=\sum_m f_m$ in $L^1$, every finite sum $\sum_{m\leq M}f_m$ is smooth and compactly supported, and \[ T_{\sum_{m\leq M}f_m}\longrightarrow T_f \] strongly but not in operator norm.

\subsection*{Antecedents and the sharpness of Zorboska's criterion}
The Fock space analogue is already contained in a family of examples of Berger and Coburn \cite[Section~6]{BergerCoburn1994}, although they use it only to probe boundedness. On $\mathbb C$ they consider $g_\lambda(z)=e^{\lambda|z|^2}$ with $\operatorname{Re}\lambda<1/4$, for which $T_{g_\lambda}$ is diagonal with eigenvalues $(1-2\lambda)^{-(k+1)}$ and $\tilde g_\lambda^{(1/2)}(w)=(1-2\lambda)^{-1} \exp\{\lambda(1-2\lambda)^{-1}|w|^2\}$. If $1-2\lambda=e^{i\theta}$ with $0<|\theta|<\pi/3$, then $T_{g_\lambda}$ is unitary on the Fock space (hence noncompact on the Fock space), while its Berezin transform decays like $\exp\{-\tfrac12(1-\cos\theta)|w|^2\}$ at infinity.
Bauer and Isralowitz asked whether similar examples exist on the weighted or unweighted Bergman space of the ball, with $f\in L^1(\B_n,dv_\gamma)$ in the weighted case \cite[Section~7]{BauerIsralowitz2012}.  Theorem~\ref{thm:main} answers this question for every $n\geq1$ and $\gamma>-1$. We next recall two positive results of Zorboska on the disk. Theorem~\ref{thm:main} settles one problem she left open, while our second theorem shows that the exponent in her localization criterion is sharp.

On the unweighted Bergman space $A^2(\mathbb D)$, Zorboska proved that boundary vanishing of $B_0(f)$ implies compactness of $T_f$ whenever $f\in\mathrm{BMO}^1(\mathbb D)$ \cite[Theorem~3.1]{Zorboska2003}.  She then proved the same conclusion for $f\in L^1(\mathbb D,dv_0)$ when $T_f$ is bounded and the two bounds in her condition~(4.3) hold.  In our notation, these bounds are
\begin{equation}\label{eq:zorboska-condition}
\sup_{z\in\mathbb D}
\norm{T_{f\circ\varphi_z}1}_{L^p(\mathbb D,dv_0)}
+
\sup_{z\in\mathbb D}
\norm{T_{\overline f\circ\varphi_z}1}_{L^p(\mathbb D,dv_0)}
<\infty,
\end{equation}
where $p>3$ and
\[
\varphi_z(w)=\frac{z-w}{1-\overline z w}
\]
\cite[Theorem~4.2, condition~(4.3)]{Zorboska2003}.  On p.~2944 she observed that compactness was unknown for an arbitrary bounded Toeplitz operator with an $L^1$ symbol and boundary vanishing of the Berezin transform.  She also asked whether \eqref{eq:zorboska-condition} suffices for some $p>2$ \cite[p.~2944, Question]{Zorboska2003}.

The case $n=1$, $\gamma=0$ of Theorem~\ref{thm:main} settles the unrestricted $L^1$ problem negatively: such an operator need not be compact. For general bounded operators, Miao and Zheng had already shown that the endpoint $p=3$ does not suffice \cite[Example~2]{MiaoZheng2004}.  Their operator is the lacunary projection used by Axler and Zheng, which is not a Toeplitz operator with an $L^1$ symbol \cite[p.~2931]{Zorboska2003}.  The next theorem gives a Toeplitz example at the endpoint.

\begin{theorem}\label{thm:zorboska-endpoint}
There exists a real function $f\in L^1(\mathbb D,dv_0)$ whose Toeplitz form has a bounded extension $T_f$ on $A^2(\mathbb D)$.  The operator $T_f$ is positive and noncompact, and
\[
B_0(f)(z)\longrightarrow0  \qquad\text{as }|z|\to1.
\]
Moreover, \eqref{eq:zorboska-condition} holds with $p=3$.  
\end{theorem}

Since $dv_0$ is a probability measure, the $L^3$ bounds imply the corresponding $L^p$ bounds for every $2<p\leq3$.  Thus the strict restriction $p>3$ in \cite[Theorem~4.2]{Zorboska2003} is sharp, and moreover the unnumbered question on p.~2944 has a negative answer.

The condition above also has an operator localization interpretation. For Toeplitz operators, automorphism covariance and $T_f^*=T_{\overline f}$ show that the two bounds in \eqref{eq:zorboska-condition}, when they hold for some $p>3$, are what Sadeghi and Zorboska call sufficient localization. They proved that every sufficiently localized operator on $A^2(\mathbb D)$ is weakly localized \cite[Section~3]{SadeghiZorboska2020}. The operator in Theorem~\ref{thm:zorboska-endpoint}, however, is not weakly localized: its Berezin transform vanishes and it is noncompact, so it lies outside $\mathcal T_0$, whereas weakly localized operators belong to $\mathcal T_0$ \cite[Theorem~1.3]{Xia2015}. Thus the strict threshold $p>3$ is also sharp for the implication from the two-sided orbit bounds \eqref{eq:zorboska-condition} to weak localization.

\subsection*{Further consequences}

Per\"al\"a, Taskinen, and Virtanen posed related boundedness and compactness problems for Toeplitz operators with integrable or distributional symbols. Let $\nu(z)=1-|z|^2$, and let $W_\nu^{-m,\infty}(\mathbb D)$ be the weighted negative Sobolev space introduced in \cite[Definition~2.3]{PeralaTaskinenVirtanen2011}. They asked, in particular, whether membership in one of these spaces for some finite $m$ is necessary for boundedness of a Toeplitz operator with an $L^1$ symbol \cite[Problem~1]{PeralaTaskinenVirtanen2011Problems}. Proposition~\ref{prop:sobolev-obstruction} shows that, for $dv_0$ denoting normalized area measure, whenever
\[
g\in L^1(\mathbb D,dv_0)\cap W_\nu^{-m,\infty}(\mathbb D),
\]
the Toeplitz form has a bounded extension on $A^2(\mathbb D)$ and
\[
T_g\text{ is compact on }A^2(\mathbb D)
\quad\Longleftrightarrow\quad
B_0(g)(z)\longrightarrow0
\qquad\text{as }|z|\to 1.
\]
It follows that the disk symbols constructed here lie outside $W_\nu^{-m,\infty}(\mathbb D)$ for every finite $m$, giving a negative answer to this specific necessity question on $A^2(\mathbb D)$. Rozenblum and Vasilevski developed a broader sesquilinear form framework for Bergman space Toeplitz operators, encompassing measure and distributional symbols, nonlocal forms, and certain forms of infinite differential order \cite{RozenblumVasilevski2016}. Their general representation theorems concern this enlarged class of symbols and do not by themselves settle the preceding necessity question for standard Toeplitz forms induced by scalar-valued functions in $L^1(\mathbb D,dv_0)$. Since the distributional Berezin transform agrees with $B_0(g)$ for the globally integrable symbols above, the proposition also supplies on $A^2(\mathbb D)$ the compactness implication behind \cite[Problem~8]{PeralaTaskinenVirtanen2011Problems}, although it does not produce the coefficient representation requested there. 

Per\"al\"a, Taskinen, and Virtanen also ask for a necessary and sufficient condition for compactness on $A^p(\mathbb D)$, $1<p<\infty$ (at least for $p=2$), with $L^1_{\mathrm{loc}}$ or distributional symbols \cite[Problem~3]{PeralaTaskinenVirtanen2011Problems}. Theorem~\ref{thm:main} shows that boundary vanishing alone is not sufficient even in the $p=2$ subcase with a globally integrable symbol. 
Taskinen and Virtanen characterize compactness on $A^p(\mathbb D)$, $1<p<\infty$, for $L^1$ symbols satisfying their weak averaging condition~(1.3). Under that condition, compactness is equivalent to \[ \norm{P(f\circ\varphi_\lambda)}_{L^q(\mathbb D,dv_0)} \longrightarrow0 \qquad\text{as }|\lambda|\to1 \] for some, equivalently every, $1\leq q<\infty$, where $P$ is the Bergman projection \cite[Theorem~3.1]{TaskinenVirtanen2018Compactness}. For the disk instance of Theorem~\ref{thm:main}, covariance and the sequence $(a_m)$ used in its proof give \[ \norm{P(f\circ\varphi_{a_m})}_{A^2(\mathbb D)} = \norm{T_fk_{a_m}}_{A^2(\mathbb D)} \not\longrightarrow0. \] Thus their norm vanishing condition fails, but the present construction does not determine whether their weak averaging condition itself holds.

For $n\geq2$, which is the standing assumption in \cite{HeCao2013}, Theorem~\ref{thm:main} also shows that the compactness criterion in \cite[Theorem~2.8]{HeCao2013} does not hold in its stated generality: the symbol furnished by the theorem satisfies the stated hypotheses, but $T_f$ is not compact. We discuss this point further in Section~\ref{sec:concluding}.

For radial symbols, Toeplitz operators are diagonal and their eigenvalues are weighted moments of the symbol \cite{GrudskyVasilevski2001,GrudskyKarapetyantsVasilevski2003}.  For bounded radial symbols on the disk, Korenblum and Zhu proved that boundary vanishing of the Berezin transform is equivalent to compactness \cite{KorenblumZhu1995}. A different limitation occurs even for bounded real-valued radial symbols: a strictly positive limit inferior of the Berezin transform does not imply positivity of the operator even modulo compact perturbations \cite{Looi2026Radial}. Our constructions use nonradially arranged blocks and do not address the radial $L^1$ case.

\subsection*{Outline of the proofs}

The proof of Theorem~\ref{thm:main} uses the following off-diagonal construction.  Let $e_K(z)=\beta_K z_1^K$ be a normalized monomial and define the rank-one operator
\[
R_Kg=\inner{g}{1}e_K.
\]
Then $\norm{R_K}=1$, whereas
\[
\norm{B_\gamma(R_K)}_{L^\infty(\B_n)}\longrightarrow0
\qquad\text{as }K\longrightarrow\infty.
\]
We convert these model operators into an $L^1$ symbol as follows.  We first approximate each $R_K$ in operator norm by a Toeplitz operator with a smooth compactly supported symbol.  This amounts to replacing a derivative of a point mass by a smooth approximation.  For finite-rank results with measure or distributional symbols, see \cite{Luecking2008,AlexandrovRozenblum2009}; by \cite[Theorem~4.1]{AlexandrovRozenblum2009}, a compactly supported distributional symbol giving a finite-rank Toeplitz operator is necessarily a finite sum of derivatives of point masses, which is the form of the symbol below. For boundedness and compactness of Toeplitz operators with distributional symbols, see \cite{PeralaTaskinenVirtanen2011}. 

We then transport the approximating symbols toward the boundary by automorphisms of the ball.  Covariance preserves the operator norms, while the $L^1$ masses can be made summable.  We choose the transport points so that the initial and final vectors of the transported rank-one blocks form Bessel sequences.  This separation makes the sum bounded and preserves its action on the weakly null sequence $(k_{a_m})$. The summable uniform bounds on the Berezin transforms of the blocks give boundary vanishing.

The proof of Theorem~\ref{thm:zorboska-endpoint} requires a different construction. Indeed, for the off-diagonal symbol constructed in the proof of Theorem~\ref{thm:main}, the matrix coefficient used to prove noncompactness also gives
\[
\norm{T_{f\circ\varphi_{a_m}}1}_{L^p(dv_0)}
\gtrsim_p K_m^{1/2-1/p}
\]
for every $p>2$, where $K_m\to\infty$ are the block degrees.  Thus that construction cannot satisfy Zorboska's additional hypothesis. Instead, we use the diagonal rank-one projections
\[
P_K=e_K\otimes e_K,
\qquad
e_K(z)=\sqrt{K+1}\,z^K.
\]
Their Berezin transforms have uniform norm of order $K^{-1}$. For $a\in\mathbb D$, with $r=|a|$, define the localized $L^3$ profile
\[
L_K(r)
= \norm{U_aP_KU_a1}_{L^3(dv_0)}.
\]
This depends only on $r$. The profiles are bounded uniformly in $K$ and, for each fixed $K$, tend to zero as $r\to0$ or $r\to1$. We choose the degrees $\ell_m$ so that
\[
\sum_m\frac{1}{\ell_m+1}<\infty
\]
and set $\eta_m=2^{-m}$. The sets
\[
C_m
= \left\{
c\in\mathbb D:
L_{\ell_m}(|c|)\geq\eta_m
\right\}
\]
are compact subsets of $\mathbb D$. We choose the transport points $b_m$ so that the transported monomials $U_{b_m}e_{\ell_m}$ form a Bessel sequence and the sets $\varphi_{b_m}(C_m)$ are pairwise disjoint. The summability of $(\ell_m+1)^{-1}$ gives boundary vanishing. The Bessel condition makes the positive sum bounded, while its action on the weakly null transported monomials makes it noncompact. The disjointness ensures that at each localization point at most one block reaches its threshold; the remaining contributions are controlled by $\sum_m\eta_m$, giving the uniform $L^3$ estimate. Finally, for each $K$ we construct a real radial symbol $h_K\in C_c^\infty(\mathbb D)$ such that
\[
T_{h_K}=P_K+E_K,
\qquad E_K\geq0,
\]
with $\norm{E_K}_{\mathcal S_1}$ arbitrarily small.  Transporting these symbols toward the boundary makes their $L^1$ masses summable and produces the symbol in Theorem~\ref{thm:zorboska-endpoint}.

\section{Preliminaries}
For $a\in\B_n$, let $\varphi_a$ be the standard involutive automorphism of $\B_n$ that interchanges $0$ and $a$, so that $\varphi_a(0)=a$ and $\varphi_a(a)=0$. The operator
\[
(U_ag)(z)
=g(\varphi_a(z))
\frac{(1-|a|^2)^{N/2}}{(1-\inner{z}{a})^N}
\]
satisfies
\[
U_a1=k_a.
\]
We will use the following transformation rule, 
\[
\int_{\B_n}F(\varphi_a(z))\dvgamma(z)
=\int_{\B_n}F(w)|k_a(w)|^2\dvgamma(w),
\]
valid for every $F$ that is nonnegative and measurable (with both sides possibly infinite), and hence for every integrable $F$; see \cite[Proposition~1.13, p.~16]{Zhu2005}. For the formula for the reproducing kernel itself, see \cite[Theorem~2.7, p.~44]{Zhu2005}. The standard identity $1-\inner{\varphi_a(z)}{a}=(1-\abs{a}^2)/(1-\inner{z}{a})$ (see \cite{Zhu2005}) yields
\begin{equation}\label{eq:kernel-identity}
k_a(\varphi_a(z))\,k_a(z)=1,
\qquad z\in\B_n.
\end{equation}
Thus $U_a^2=I$, since $\varphi_a\circ\varphi_a =\operatorname{id}$. Moreover, the transformation rule applied to $F=\abs{g}^2\abs{k_a\circ\varphi_a}^2$, together with \eqref{eq:kernel-identity}, gives $\norm{U_ag}=\norm{g}$ for every $g\in A^2_\gamma$. Hence $U_a$ is a self-adjoint unitary on $A^2_\gamma$, since $U_a^*=U_a^*U_a^2=U_a$.

\begin{lemma}\label{lem:l1-covariance}
Suppose that $h\in L^1(\B_n,dv_\gamma)$ and that its Toeplitz form has a bounded extension $T_h$. Then
\[
\inner{T_hg_1}{g_2}
=\int_{\B_n}h(z)g_1(z)\overline{g_2(z)}\dvgamma(z)
\]
for all bounded holomorphic functions $g_1$ and $g_2$. Moreover, for every $a\in\B_n$, the Toeplitz form with symbol
$h\circ\varphi_a$ has the bounded extension
\[
T_{h\circ\varphi_a}=U_aT_hU_a^*.
\]
\end{lemma}

\begin{proof}
For a bounded holomorphic function $g$, let $g_r(z):=g(rz)$. For fixed $r<1$, the Taylor polynomials of $g_r$ converge uniformly on $\overline{\B_n}$. Hence the defining identity for $T_h$, first for polynomials and then by uniform approximation, gives
\[
\inner{T_h(g_1)_r}{(g_2)_r}
=\int_{\B_n}h(z)(g_1)_r(z)\overline{(g_2)_r(z)}\dvgamma(z).
\]
As $r\to1$, $(g_j)_r\to g_j$ in $A^2_\gamma$ by dominated convergence, so the left-hand side converges because $T_h$ is bounded. The right-hand side also converges by dominated convergence, since $|(g_j)_r|\le\norm{g_j}_\infty$. This proves the first assertion.

For fixed $a\in\B_n$, the transformation rule and the boundedness of $k_a$ give
\[
\norm{h\circ\varphi_a}_{L^1(dv_\gamma)}
=\int_{\B_n}|h(w)||k_a(w)|^2\dvgamma(w)<\infty.
\]
For holomorphic polynomials $p,q$, the functions $U_ap$ and $U_aq$ are bounded and holomorphic for fixed $a \in \mathbb B_n$. Hence the first assertion, the transformation rule, and the fact that $U_a$ is a self-adjoint unitary ($U_a^2=I$) give
\[
\begin{aligned}
\inner{U_aT_hU_a^*p}{q}
&=\inner{T_hU_ap}{U_aq}\\
&=\int_{\B_n}h(w)
  p(\varphi_a(w))\overline{q(\varphi_a(w))}
  |k_a(w)|^2\dvgamma(w)\\
&=\int_{\B_n}h(\varphi_a(z))p(z)\overline{q(z)}
  \dvgamma(z).
\end{aligned}
\]
Thus $U_aT_hU_a^*$ is the bounded extension of the Toeplitz form with symbol $h\circ\varphi_a$.
\end{proof}

For every bounded operator $A$ on $A^2_\gamma$, we also have
\begin{equation}\label{eq:berezin-covariance}B_\gamma(U_aAU_a^*)=B_\gamma(A)\circ\varphi_a.\end{equation}
Indeed, since $U_a$ is a self-adjoint unitary, $B_\gamma(U_aAU_a^*)(z)=\inner{AU_ak_z}{U_ak_z}$. The definition of $U_a$ shows that $U_ak_z=\lambda\,k_{\varphi_a(z)}$ for a scalar $\lambda=\lambda(a,z)$, and $\abs{\lambda}=\norm{U_ak_z}=1$, which gives \eqref{eq:berezin-covariance}. (We note that while $|\lambda|=1$ always, its phase varies with $\langle z,a\rangle$.)

\begin{lemma}\label{lem:weak-null}
For every $w\in A^2_\gamma$,
\[
U_aw\rightharpoonup 0
\qquad\text{as } |a|\to 1.
\]
In particular, $k_a\rightharpoonup 0$ as $|a|\to 1$.
\end{lemma}

\begin{proof}
Since $U_a$ is unitary and the holomorphic polynomials are dense in $A^2_\gamma$, it suffices to prove that $\inner{U_ap}{q}\to 0$ as $|a|\to 1$ for all holomorphic polynomials $p$ and $q$. Indeed, for general $w,g\in A^2_\gamma$ and polynomials $p,q$,
\[
\abs{\inner{U_aw}{g}}
\le\norm{w-p}\,\norm{g}
+\norm{p}\,\norm{g-q}
+\abs{\inner{U_ap}{q}}.
\]
Fix polynomials $p,q$ and $0<r<1$. If $z\in r\B_n$, then $\abs{1-\inner{z}{a}}\ge 1-r$, so the formula for $U_a$ gives
\[
\sup_{z\in r\B_n}\abs{(U_ap)(z)}
\le\Big(\sup_{\B_n}\abs{p}\Big)(1-r)^{-N}(1-|a|^2)^{N/2}
=:\delta_r(a),
\]
and $\delta_r(a)\to 0$ as $|a|\to 1$. Splitting the integral defining $\inner{U_ap}{q}$ over $r\B_n$ and $\B_n\setminus r\B_n$, and applying
the Cauchy--Schwarz inequality on the second region,
\[
\abs{\inner{U_ap}{q}}
\le\delta_r(a)\norm{q}_{L^1(dv_\gamma)}
+\norm{p}_{A^2_\gamma}
\left(\int_{\B_n\setminus r\B_n}\abs{q}^2\dvgamma\right)^{1/2}.
\]
Letting $|a|\to 1$ and then $r\to 1$ proves the lemma.
\end{proof}

\section{Three inputs to the construction}

\subsection{Rank-one operators with small Berezin transform}

For an integer $K\ge 0$, define
\[
\beta_K
=\left(\frac{\Gamma(N+K)}{\Gamma(N)K!}\right)^{1/2},
\qquad
e_K(z)=\beta_K z_1^K.
\]
Then $\norm{e_K}_{A^2_\gamma}=1$. Let
\[
R_Kh:=\inner{h}{1}e_K.
\]
Thus $R_K$ has rank one and $\norm{R_K}=1$.

For vectors $x,y$ in a Hilbert space, $x\otimes y$ denotes the rank-one operator $(x\otimes y)\, h := \langle h,y\rangle \, x$. In this notation, $R_K = e_K \otimes 1$. 

\begin{lemma}\label{lem:rank-one-berezin}
There is a constant $C_N>0$ such that
\[
\norm{B_\gamma(R_K)}_\infty
\le C_NK^{-(N+1)/2}
\]
for every $K\ge 1$.
\end{lemma}

\begin{proof}
For $z\in\B_n$,
\[
B_\gamma(R_K)(z)
=\inner{k_z}{1}\inner{e_K}{k_z}.
\]
Thus, writing $r=|z|$,
\[
\abs{B_\gamma(R_K)(z)}
=\beta_K(1-|z|^2)^N|z_1|^K
\le \beta_K(1-r^2)^Nr^K.
\]
The maximum of $(1-r^2)^Nr^K$ occurs at
$r^2=K/(K+2N)$. Therefore
\[
\sup_{0\le r<1}(1-r^2)^Nr^K
=\left(\frac{2N}{K+2N}\right)^N
\left(\frac{K}{K+2N}\right)^{K/2}
\le C_NK^{-N}.
\]
Stirling's formula gives $\beta_K\le C_NK^{(N-1)/2}$, and the result follows.
\end{proof}

\subsection{Approximation by smooth Toeplitz symbols}

Finite-rank Bergman--Toeplitz operators with compactly supported measure or distributional symbols are treated in \cite{Luecking2008,AlexandrovRozenblum2009}.

\begin{lemma}\label{lem:rank-one-realization}
For every $K\ge 0$ and every $\eta>0$, there is a function $h_{K,\eta}\in C_c^\infty(\B_n)$ such that
\[
\norm{T_{h_{K,\eta}}-R_K}_{A^2_\gamma\to A^2_\gamma}<\eta.
\]
\end{lemma}

\begin{proof}
Let $D_K:=\partial^K/\partial\overline z_1^{\,K}$. Throughout the proof, $C$ denotes a positive constant depending only on $K$, $n$, and $N$, whose value may change from one occurrence to the next. The proof runs as follows. We first show that the Toeplitz form of $R_K$ evaluates the function $\frac{1}{K!\beta_K}D_K(p\overline q)$ at the origin, so that the symbol of $R_K$ is a constant multiple of $D_K\delta_0$. We then replace the point mass by a smooth bump. The resulting error is an average of a function minus its value at the center of the average, and we estimate it by a Lipschitz bound that holds uniformly over the unit balls of $A^2_\gamma$.

\medskip
\noindent
\emph{Step 1: the Toeplitz form of $R_K$ is a point evaluation.} We claim that for all holomorphic polynomials $p$ and $q$,
\begin{equation}\label{eq:point-eval}
\inner{R_Kp}{q}
=\frac{1}{K!\beta_K}\,D_K(p\overline q)(0).
\end{equation}
Distinct monomials are orthogonal in $A^2_\gamma$, since $dv_\gamma$ is invariant under rotations in each coordinate. Hence only the constant term of $p$ contributes to $\inner{p}{1}$, and $\inner{p}{1}=p(0)$ because $dv_\gamma$ is a probability measure. We write $c_K$ for the coefficient of $z_1^K$ in $q$. Only the $z_1^K$ term of $q$ pairs nontrivially with $e_K$, so
\[
\inner{e_K}{q}
=\beta_K\,\overline{c_K}\,\norm{z_1^K}_{A^2_\gamma}^2
=\frac{\overline{c_K}}{\beta_K}
=\frac{1}{K!\beta_K}\,
\overline{\frac{\partial^Kq}{\partial z_1^K}(0)},
\]
where we used $\norm{z_1^K}_{A^2_\gamma}=\beta_K^{-1}$, which is the normalization of $e_K$, together with $c_K=\frac{1}{K!}\frac{\partial^Kq}{\partial z_1^K}(0)$. Therefore
\[
\inner{R_Kp}{q}
=\inner{p}{1}\inner{e_K}{q}
=\frac{1}{K!\beta_K}\,
p(0)\,\overline{\frac{\partial^Kq}{\partial z_1^K}(0)}.
\]
On the other hand, $\partial/\partial\overline z_1$ annihilates the holomorphic factor $p$, so
\begin{equation}\label{eq:DK-product}
D_K(p\overline q)
=p\,\overline{\frac{\partial^Kq}{\partial z_1^K}},
\end{equation}
and evaluating \eqref{eq:DK-product} at the origin proves \eqref{eq:point-eval}.

\medskip
\noindent
\emph{Step 2: mollification.}
We fix $\rho\in C_c^\infty(\C^n)$ with $\rho \ge 0$, $\operatorname{supp}\rho\subset\frac12\B_n$ and $\int_{\C^n}\rho\dv=1$, and for $0<\epsilon\le1$ we put
\[
\rho_\epsilon(z)=\epsilon^{-2n}\rho(z/\epsilon),
\qquad
h_\epsilon(z)
=\frac{(-1)^K}{K!\beta_Kc_\gamma}\,
(1-\abs{z}^2)^{-\gamma}\,D_K\rho_\epsilon(z).
\]
Then $\operatorname{supp}h_\epsilon =\operatorname{supp} D_K(\rho_\epsilon) \subset \operatorname{supp}\rho_\epsilon\subset\frac{\epsilon}{2}\B_n$, and $h_\epsilon\in C_c^\infty(\B_n)$ because the weight $(1-\abs z^2)^{-\gamma}$ is smooth on $\B_n$. Since $h_\epsilon$ is bounded with compact support, the Toeplitz operator $T_{h_\epsilon}$ is bounded on $A^2_\gamma$ and satisfies $\inner{T_{h_\epsilon}p}{q} =\int_{\B_n}h_\epsilon\,p\overline q\dvgamma$ for all holomorphic polynomials $p,q$.

We next verify that, for every $\psi\in C^\infty(\B_n)$,
\begin{equation}\label{eq:mollified-form}
\int_{\B_n}h_\epsilon\,\psi\dvgamma
=\frac{1}{K!\beta_K}
\int_{\B_n}\rho_\epsilon\,D_K\psi\dv.
\end{equation}
Indeed, $\dvgamma=c_\gamma(1-\abs z^2)^\gamma\dv$, so the weight factors cancel and the left side equals
\[
\frac{(-1)^K}{K!\beta_K}
\int_{\B_n}(D_K\rho_\epsilon)\,\psi\dv.
\]
Since $\partial/\partial\overline z_1 =\frac12(\partial/\partial x_1+i\,\partial/\partial y_1)$ and $\rho_\epsilon$ has compact support, we may integrate by parts $K$ times in $\overline z_1$ without boundary terms.  Each integration contributes one factor of $-1$, and \eqref{eq:mollified-form} follows.

We now take $\psi=p\overline q$ in \eqref{eq:mollified-form}, subtract \eqref{eq:point-eval}, and abbreviate
\[
F=F_{p,q}
:=\frac{1}{K!\beta_K}D_K(p\overline q)
=\frac{1}{K!\beta_K}\,
p\,\overline{\frac{\partial^Kq}{\partial z_1^K}}.
\]
This gives
\begin{equation}\label{eq:avg-minus-point}
\inner{(T_{h_\epsilon}-R_K)p}{q}
=\int_{\C^n}\rho_\epsilon F\dv-F(0)
=\int_{\C^n}\rho_\epsilon(z)
\bigl(F(z)-F(0)\bigr)\dv(z),
\end{equation}
where the last equality uses $\int_{\C^n}\rho_\epsilon\dv=1$.

\medskip
\noindent
\emph{Step 3: a Lipschitz bound for $F$, uniform in $p$ and $q$.}
From now on we take $0<\epsilon\le\frac12$, so that $\operatorname{supp}\rho_\epsilon\subset\frac{\epsilon}{2}\B_n \subset\frac14\B_n$. Suppose $\norm{p}_{A^2_\gamma}\le1$ and $\norm{q}_{A^2_\gamma}\le1$. Every $g\in A^2_\gamma$ satisfies
\[
\abs{g(w)}
=\abs{\inner{g}{K_w}}
\le\norm{g}_{A^2_\gamma}(1-\abs w^2)^{-N/2},
\]
so $\abs p\le C$ and $\abs q\le C$ on $\frac12\B_n$. We use the Cauchy estimates on polydiscs: if $g$ is holomorphic on a neighbourhood of the closed polydisc $\overline D(z,r)=\{w:\abs{w_j-z_j}\le r,\ 1\le j\le n\}$, then, for every multiindex $\alpha$,
\begin{equation}\label{eq:cauchy-est}
\abs{\partial^\alpha g(z)}
\le\frac{\alpha!}{r^{\abs\alpha}}
\sup_{\overline D(z,r)}\abs g,
\end{equation}
as follows by iterating the one-variable Cauchy integral formula in each coordinate. If $z\in\frac14\B_n$ and $r=\frac{1}{4\sqrt n}$, then every $w\in\overline D(z,r)$ satisfies $\abs{w-z}\le\sqrt n\,r=\frac14$, so $\overline D(z,r)\subset\frac12\B_n$. Hence \eqref{eq:cauchy-est} bounds every partial derivative of $p$ of order at most $1$, and of $q$ of order at most $K+1$, by $C$ on $\frac14\B_n$. By \eqref{eq:DK-product}, $F=\frac{1}{K!\beta_K}\,p\,\overline{\partial^Kq/\partial z_1^K}$; since $p$ and $\partial^Kq/\partial z_1^K$ are holomorphic, their real first partial derivatives coincide in modulus with their complex derivatives $\partial/\partial z_j$, so the first-order real partial derivatives of $F$ are bounded by $C$ on $\frac14\B_n$. Since $\frac14\B_n$ is convex, writing $F(z)-F(0)=\int_0^1\frac{d}{dt}F(tz)\,dt$ gives
\begin{equation}\label{eq:lipschitz}
\abs{F(z)-F(0)}\le C\abs{z},
\qquad z\in\tfrac14\B_n.
\end{equation}

\medskip
\noindent
\emph{Step 4: conclusion.}
The support of $\rho_\epsilon$ lies in $\frac14\B_n$, so \eqref{eq:lipschitz} applies wherever the integrand in \eqref{eq:avg-minus-point} is nonzero. Combining \eqref{eq:avg-minus-point} and \eqref{eq:lipschitz} with the substitution $z=\epsilon w$, we obtain, uniformly over $\norm{p}_{A^2_\gamma}\le1$ and $\norm{q}_{A^2_\gamma}\le1$,
\[
\abs{\inner{(T_{h_\epsilon}-R_K)p}{q}}
\le C\int_{\C^n}\abs{\rho_\epsilon(z)}\,\abs{z}\dv(z)
=C\epsilon\int_{\C^n}\abs{\rho(w)}\,\abs{w}\dv(w)
\le C\epsilon.
\]
Since $T_{h_\epsilon}-R_K$ is bounded and the holomorphic polynomials are dense in $A^2_\gamma$,
\[
\norm{T_{h_\epsilon}-R_K}
=\sup_{\substack{p,q\ \mathrm{polynomials}\\
\norm{p}\le 1,\ \norm{q}\le 1}}
\abs{\inner{(T_{h_\epsilon}-R_K)p}{q}}
\le C\epsilon.
\]
Taking $h_{K,\eta}=h_\epsilon$ with $\epsilon<\min\{\tfrac12,\,\eta/C\}$ proves the lemma.
\end{proof}

\subsection{Moving and separating the blocks}

The next lemma will be used twice: to make the $L^1$ masses of the moved symbols summable, and to show that the Berezin transform of each block vanishes at the boundary.

\begin{lemma}[Decay lemma]\label{lem:decay}
Let $h\in L^1(\B_n,dv_\gamma)$ with $\operatorname{supp}h\subset s\B_n$ for some $s<1$. Then
\[
B_\gamma(\abs{h})(z)
\le\frac{(1-\abs{z}^2)^N}{(1-s)^{2N}}\,
\norm{h}_{L^1(dv_\gamma)},
\qquad z\in\B_n.
\]
In particular, $\norm{h\circ\varphi_a}_{L^1(dv_\gamma)} =B_\gamma(\abs{h})(a)\to 0$ as $\abs{a}\to 1$, and $B_\gamma(h)(z)\to 0$ as $\abs{z}\to 1$.
\end{lemma}

\begin{proof}
If $w\in\operatorname{supp}h$ and $z\in\B_n$, then $\abs{1-\inner{w}{z}}\ge 1-\abs{w}\abs{z}\ge 1-s$, so
\[
\abs{k_z(w)}^2
=\frac{(1-\abs{z}^2)^N}{\abs{1-\inner{w}{z}}^{2N}}
\le\frac{(1-\abs{z}^2)^N}{(1-s)^{2N}}.
\]
Multiplying by $\abs{h(w)}$ and integrating gives the stated inequality. The identity $\norm{h\circ\varphi_a}_{L^1(dv_\gamma)}=B_\gamma(\abs{h})(a)$ is the transformation rule applied to $F=\abs{h}$, and both limits follow from the inequality, since $\abs{B_\gamma(h)}\le B_\gamma(\abs{h})$.
\end{proof}

Recall that a sequence $(x_m)$ in a Hilbert space is a Bessel sequence if
\[
\sum_m\abs{\inner{h}{x_m}}^2\le C\norm{h}^2
\]
for some $C$ and every $h$.

\begin{lemma}[Separation lemma]\label{lem:separation}
Let $(s_m)$ and $(t_m)$ be sequences of unit vectors in $A^2_\gamma$, and let $r_m\to 1$ with $0<r_m<1$. There are points $a_m\in\B_n$ with $|a_m|>r_m$ such that
\[
u_m=U_{a_m}s_m,
\qquad
v_m=U_{a_m}t_m
\]
are Bessel sequences and
\[
\rho_m^u:=\sum_{j\ne m}\abs{\inner{u_m}{u_j}}^2\to 0,
\qquad
\rho_m^v:=\sum_{j\ne m}\abs{\inner{v_m}{v_j}}^2\to 0.
\]
Moreover, both sequences have Bessel bound $\tfrac53$, and the synthesis maps \[ (c_m)\mapsto\sum_m c_mu_m, \qquad (c_m)\mapsto\sum_m c_mv_m \] are bounded from $\ell^2$ into $A^2_\gamma$ with norm at most $\bigl(\tfrac53\bigr)^{1/2}$.
\end{lemma}

\begin{proof}
We choose the points $a_m$ inductively. When $m=1$, we take any $a_1\in\B_n$ with $\abs{a_1}>r_1$; for $m\ge 2$, suppose that $a_1,\ldots,a_{m-1}$ have been chosen. By Lemma~\ref{lem:weak-null}, each of the finitely many functions,
\[
a\mapsto\inner{U_as_m}{u_j},
\qquad
a\mapsto\inner{U_at_m}{v_j},
\qquad
j<m,
\]
tends to zero as $|a|\to 1$. Hence we may choose $a_m\in\B_n$ with $|a_m|>r_m$ such that
\[
\sum_{j<m}
\left(
\abs{\inner{u_m}{u_j}}^2
+\abs{\inner{v_m}{v_j}}^2
\right)
\le 4^{-m},
\]
where $u_m=U_{a_m}s_m$ and $v_m=U_{a_m}t_m$ as in the statement. Since each $U_{a_m}$ is unitary, $u_m$ and $v_m$ are unit vectors.

We estimate $\rho_m^u$ for fixed $m$. The terms with $j<m$ sum to at most $4^{-m}$ by the choice of $a_m$. If $j>m$, then $\abs{\inner{u_m}{u_j}}^2=\abs{\inner{u_j}{u_m}}^2$ appears in the sum controlled at stage $j$, so this term is at most $4^{-j}$. Therefore
\[
\rho_m^u
\le 4^{-m}+\sum_{j>m}4^{-j}
=\frac43\,4^{-m}\to 0,
\]
and the same estimate holds for $\rho_m^v$.

It remains to verify the Bessel property. Let $G^u=\big(\inner{u_m}{u_j}\big)_{j,m}$ be the Gram matrix of $(u_m)$ and write $G^u=I+E^u$, where $E^u$ contains the off-diagonal entries. Then
\[
\norm{E^u}_{\mathrm{HS}}^2
=\sum_m\rho_m^u
\le\frac43\sum_{m=1}^\infty 4^{-m}
=\frac49,
\]
so $G^u$ is bounded on $\ell^2$ with $\norm{G^u}\le 1+\norm{E^u}_{\mathrm{HS}}\le\tfrac53$. For a finitely supported scalar sequence $(c_m)$,
\[
\Big\|\sum_m c_mu_m\Big\|^2
=\inner{G^uc}{c}_{\ell^2}
\le\frac53\,\norm{c}_{\ell^2}^2,
\]
so the synthesis map $(c_m)\mapsto\sum_mc_mu_m$ extends to a bounded operator from $\ell^2$ into $A^2_\gamma$. Its adjoint is the analysis map $h\mapsto(\inner{h}{u_m})_m$, and its boundedness is the Bessel inequality
\[
\sum_m\abs{\inner{h}{u_m}}^2\le\frac53\,\norm{h}^2,
\qquad h\in A^2_\gamma.
\]
The same argument applies to $(v_m)$.
\end{proof}

\section{Construction of the symbol}

\begin{proof}[Proof of Theorem~\ref{thm:main}]
\emph{Choice of the data.}
We choose, in this order:
\begin{itemize}
\item[(i)] integers $K_m\to\infty$ with
$\norm{B_\gamma(R_{K_m})}_\infty<2^{-m-4}$
(Lemma~\ref{lem:rank-one-berezin});
\item[(ii)] functions $h_m^0\in C_c^\infty(\B_n)$ with
$\norm{T_{h_m^0}-R_{K_m}}<2^{-m-4}$
(Lemma~\ref{lem:rank-one-realization});
\item[(iii)] radii $r_m$ with $1-\tfrac1m\le r_m<1$ such that
$\norm{h_m^0\circ\varphi_a}_{L^1(dv_\gamma)}<2^{-m}$
whenever $|a|>r_m$ (Lemma~\ref{lem:decay});
\item[(iv)] points $a_m\in\B_n$ with $|a_m|>r_m$ such that
\[
u_m=U_{a_m}e_{K_m},
\qquad
v_m=U_{a_m}1=k_{a_m}
\]
are Bessel sequences of unit vectors with Bessel bound $\tfrac53$ and $\rho_m^u,\rho_m^v\to 0$ (Lemma~\ref{lem:separation}, applied with $s_m=e_{K_m}$ and $t_m=1$).
\end{itemize} 
Note that $|a_m|>r_m\ge 1-\tfrac1m$, so $|a_m|\to 1$. We set
\[
f_m=h_m^0\circ\varphi_{a_m},
\qquad
f=\sum_{m=1}^\infty f_m.
\]
By (iii) and (iv), $\sum_m\norm{f_m}_{L^1(dv_\gamma)}<1$, so the series converges in $L^1(\B_n,dv_\gamma)$ and $f\in L^1(\B_n,dv_\gamma)$.

\medskip
\noindent
\emph{The Toeplitz form of $f$ extends boundedly.}
Since $R_{K_m}=e_{K_m}\otimes 1$ by definition, Lemma~\ref{lem:l1-covariance} gives
\[
T_{f_m}=U_{a_m}T_{h_m^0}U_{a_m}^*,
\qquad
Q_m:=U_{a_m}R_{K_m}U_{a_m}^*=u_m\otimes v_m,
\]
and, because each $U_{a_m}$ is unitary, (ii) yields
\begin{equation}\label{eq:block-error}
\norm{T_{f_m}-Q_m}=\norm{T_{h_m^0}-R_{K_m}}<2^{-m-4}.
\end{equation}
By Lemma~\ref{lem:separation}, the sequence $(v_m)$ has Bessel bound $\tfrac53$ and the synthesis map of $(u_m)$ has norm at most $\bigl(\tfrac53\bigr)^{1/2}$. Hence, for every $h\in A^2_\gamma$ and $M'\ge M$,
\[
\Big\|\sum_{m=M}^{M'}\inner{h}{v_m}u_m\Big\|
\le\Big(\tfrac53\Big)^{1/2}
\Big(\sum_{m\ge M}\abs{\inner{h}{v_m}}^2\Big)^{1/2}
\xrightarrow[M\to\infty]{}0.
\]
Hence the series \[Qh:=\sum_m\inner{h}{v_m}u_m\]converges in $A^2_\gamma$, the partial sums $\sum_{m\le M}Q_m$ converge to $Q$ in the strong operator topology, and $\norm{Q}\le\tfrac53$. By \eqref{eq:block-error}, the series $E:=\sum_m(T_{f_m}-Q_m)$ converges in operator norm with $\norm{E}\le\tfrac1{16}$. Therefore
\[
\sum_{m\le M}T_{f_m}\xrightarrow[M\to\infty]{} A:=Q+E
\]
in the strong operator topology, and $\norm{A}\le\tfrac53+\tfrac1{16}$.

Next we check that $A$ is the bounded extension of the Toeplitz form with symbol $f$. If $p,q$ are holomorphic polynomials, then $p\overline q$ is bounded on
$\B_n$. Since $\sum_mf_m$ converges in $L^1$,
\[
\int_{\B_n}f\,p\overline q\dvgamma
=\sum_m\int_{\B_n}f_m\,p\overline q\dvgamma
=\sum_m\inner{T_{f_m}p}{q}
=\inner{Ap}{q}.
\]
Since the polynomials are dense in $A^2_\gamma$, $A$ is the unique
bounded operator with this property, and we write $T_f=A$.

\medskip
\noindent
\emph{The operator is not compact.}
The vectors $v_m=k_{a_m}$ have unit norm and, since $|a_m|\to 1$, Lemma~\ref{lem:weak-null} (applied with $w=1$) shows that $v_m \rightharpoonup 0$. Since $\norm{v_m}=\norm{u_m}=1$, 
\[
\inner{Qv_m}{u_m}
=1+\sum_{j\ne m}\inner{v_m}{v_j}\inner{u_j}{u_m},
\]
and the Cauchy--Schwarz inequality gives
\[
\abs{\inner{Qv_m}{u_m}-1}
\le(\rho_m^v)^{1/2}(\rho_m^u)^{1/2}\to 0.
\]
By Lemma~\ref{lem:separation},
\[
\rho_m^u,\rho_m^v\le \frac43\,4^{-m},
\]
and hence
\[
(\rho_m^u\rho_m^v)^{1/2}
\le \frac43\,4^{-m}\le \frac13.
\]
Therefore
\[
\begin{aligned}
\norm{T_fv_m}
&\ge \abs{\inner{T_fv_m}{u_m}}\\
&\ge \abs{\inner{Qv_m}{u_m}}-\norm{E}\\
&\ge 1-(\rho_m^u\rho_m^v)^{1/2}-\frac1{16}\\
&\ge \frac{29}{48}.
\end{aligned}
\]
Thus
\[
\inf_m\norm{T_fk_{a_m}}
=\inf_m\norm{T_fv_m}
\ge \frac{29}{48}>0.
\]
Since $v_m=k_{a_m}\rightharpoonup0$, whereas $\norm{T_fv_m}$ is bounded away from zero, $T_f$ is not compact.

\medskip
\noindent
\emph{The Berezin transform vanishes at the boundary.}
For each $m$, Lemma~\ref{lem:l1-covariance} gives
\[
B_\gamma(f_m)=B_\gamma(T_{f_m}).
\]
Using \eqref{eq:berezin-covariance}, the fact that $\varphi_{a_m}$ is a bijection of $\B_n$, and \eqref{eq:block-error}, we obtain
\[
\begin{aligned}
\norm{B_\gamma(f_m)}_\infty
&\le
\norm{B_\gamma(Q_m)}_\infty
+\norm{T_{f_m}-Q_m}\\
&=
\norm{B_\gamma(R_{K_m})}_\infty
+\norm{T_{f_m}-Q_m}\\
&<2^{-m-3}.
\end{aligned}
\]
Thus the series $\sum_{m=1}^\infty B_\gamma(f_m)$ converges uniformly on $\B_n$.

We next show that each term in this series vanishes at the boundary. Since $\varphi_{a_m}$ is an involution, $\operatorname{supp}f_m=\varphi_{a_m}(\operatorname{supp}h_m^0)$ is a compact subset of $\B_n$, hence contained in $s_m\B_n$ for some $s_m<1$. By Lemma~\ref{lem:decay}, $B_\gamma(f_m)(z)\to 0$ as $\abs{z}\to 1$ for each fixed $m$.

We now identify the sum of these transforms with $B_\gamma(f)$. For fixed $z\in\B_n$, the function $|k_z|^2$ is bounded on $\B_n$; indeed,
\[
\sup_{w\in\B_n}|k_z(w)|^2
\le
\frac{(1-|z|^2)^N}{(1-|z|)^{2N}}
=
\left(\frac{1+|z|}{1-|z|}\right)^N.
\]
Since $\sum_m f_m$ converges to $f$ in $L^1(\B_n,dv_\gamma)$,
\[
\begin{aligned}
\abs{
B_\gamma(f)(z)-\sum_{m=1}^M B_\gamma(f_m)(z)
}
&\le
\left(\sup_{w\in\B_n}|k_z(w)|^2\right)
\norm{f-\sum_{m=1}^M f_m}_{L^1(dv_\gamma)}\\
&\longrightarrow 0.
\end{aligned}
\]
Thus
\[
B_\gamma(f)(z)
=
\sum_{m=1}^\infty B_\gamma(f_m)(z)
\qquad (z\in\B_n).
\]

Finally, we fix $\epsilon>0$. We choose $M$ so large that
\[
\sum_{m>M}\norm{B_\gamma(f_m)}_\infty
<\frac{\epsilon}{2}.
\]
Since each of the finitely many functions $B_\gamma(f_1),\ldots,B_\gamma(f_M)$ vanishes at the boundary, there is an $r<1$ such that
\[
\sum_{m=1}^M\abs{B_\gamma(f_m)(z)}
<\frac{\epsilon}{2}
\qquad\text{whenever } |z|>r.
\]
For such $z$,
\[
\begin{aligned}
\abs{B_\gamma(f)(z)}
&\le
\sum_{m=1}^M\abs{B_\gamma(f_m)(z)}
+
\sum_{m>M}\norm{B_\gamma(f_m)}_\infty\\
&<\epsilon.
\end{aligned}
\]
In conclusion, $B_\gamma(f)(z)\to0$ as $|z|\to1.$
\end{proof}

\section{Sharpness of Zorboska's exponent}
\label{sec:zorboska}

Throughout this section, we let $n=1$ and $\gamma=0$, and we write $dA=dv_0$ for normalized area measure on $\mathbb D$. Thus
\[
e_K(z)=\sqrt{K+1}\,z^K,
\qquad K\ge 0,
\]
is the standard orthonormal basis of $A^2(\mathbb D)$.

We first explain why the off-diagonal construction used in Theorem~\ref{thm:main} does not answer Zorboska's more restrictive question.

\begin{proposition}\label{prop:offdiagonal-localization}
For the symbol $f$ and the points $(a_m)$ constructed in the proof of Theorem~\ref{thm:main},
\[
\sup_{a\in\mathbb D}
\norm{T_{f\circ\varphi_a}1}_{L^p(dA)}=\infty
\]
for every $p>2$.
\end{proposition}

\begin{proof}
We retain the notation
\[
u_m=U_{a_m}e_{K_m},
\qquad
v_m=k_{a_m}
\]
from the proof of Theorem~\ref{thm:main}. That proof gives a constant $c>0$ such that
\[
\abs{\inner{T_fv_m}{u_m}}\ge c
\]
for all sufficiently large $m$. By covariance and the self-adjointness of $U_{a_m}$,
\[
G_m:=T_{f\circ\varphi_{a_m}}1
=U_{a_m}T_fv_m
\]
satisfies
\[
\abs{\inner{G_m}{e_{K_m}}}
=\abs{\inner{T_fv_m}{u_m}}
\ge c.
\]
Let $p'=p/(p-1)$. If $G_m\notin L^p(dA)$, there is nothing to prove. Otherwise, H\"older's inequality gives
\[
c\le \norm{G_m}_{L^p(dA)}
\norm{e_{K_m}}_{L^{p'}(dA)}.
\]
For every $q>0$,
\[
\norm{e_K}_{L^q(dA)}^q
=(K+1)^{q/2}\frac{2}{qK+2}.
\]
Since $p'<2$, it follows that
\[
\norm{G_m}_{L^p(dA)}
\gtrsim_p K_m^{1/2-1/p}\longrightarrow\infty. \qedhere
\]
\end{proof}

\subsection{Localized diagonal blocks}

For $K\ge1$, let
\[
P_K=e_K\otimes e_K,
\qquad \text{ so that }
P_Kg=\inner{g}{e_K}e_K.
\]
For a bounded operator $A$ on $A^2(\mathbb D)$, let
\[
\Lambda_p(A)
:=\sup_{a\in\mathbb D}
\norm{U_aAU_a1}_{L^p(dA)}
\]
whenever the right-hand side is finite.

We will repeatedly use the following consequence of the cocycle identity for the automorphic unitaries. Suppose that $A$ commutes with rotations and $b\in\mathbb D$. Then
\begin{equation}\label{eq:radial-localization-covariance}
\norm{U_a(U_bAU_b)U_a1}_{L^p(dA)}
=\norm{U_cAU_c1}_{L^p(dA)},
\qquad |c|=|\varphi_b(a)|.
\end{equation}
Indeed, a direct calculation from the definition of $U_a$ shows that $U_bU_a=\lambda VU_c$, where $|\lambda|=1$ and $V$ is a rotation. Since $U_aU_b=(U_bU_a)^*$, rotations act isometrically on $L^p(dA)$, and $V^*AV=A$, identity \eqref{eq:radial-localization-covariance} follows.

\begin{lemma}\label{lem:diagonal-block-profile}
Let $K\ge1$ and $a\in\mathbb D$, and set $r=|a|$. Then
\begin{equation}\label{eq:diagonal-L3-profile}
\begin{aligned}
L_K(r)^3
&:=\norm{U_aP_KU_a1}_{L^3(dA)}^3\\
&=(K+1)^3r^{3K}(1-r^2)^2
\left(\frac{2}{3K+2}+\frac{2r^2}{3K+4}\right).
\end{aligned}
\end{equation}
Therefore,
\[
\sup_{K\ge1}\Lambda_3(P_K)<\infty,
\]
and, for fixed $K$, $L_K(r)\to0$ as $r\to0$ or $r\to1$. Moreover,
\begin{equation}\label{eq:diagonal-berezin-bound}
\norm{B_0(P_K)}_\infty\le\frac{4}{K+1}.
\end{equation}
\end{lemma}

\begin{proof}
Since $U_a1=k_a$,
\[
P_KU_a1=\inner{k_a}{e_K}e_K,
\qquad
\abs{\inner{k_a}{e_K}}
=\sqrt{K+1}(1-r^2)r^K.
\]
Using the formula for $U_a$, followed by the change of variables $w=\varphi_a(v)$, gives
\begin{align*}
\norm{U_aP_KU_a1}_{L^3(dA)}^3
&=(K+1)^3r^{3K}(1-r^2)^2
\int_{\mathbb D}|v|^{3K}|1-\overline av|^2\,dA(v)\\
&=(K+1)^3r^{3K}(1-r^2)^2
\left(\frac{2}{3K+2}+\frac{2r^2}{3K+4}\right),
\end{align*}
because the mixed terms integrate to zero. This proves \eqref{eq:diagonal-L3-profile}.

The expression on the right is bounded by a constant multiple of
\[
(K+1)^2(1-r^2)^2r^{3K},
\]
which is uniformly bounded in $K$ and $r$. Its vanishing at $r=0$ and $r=1$ for fixed $K$ is immediate.

Finally,
\[
B_0(P_K)(z)
=(K+1)(1-|z|^2)^2|z|^{2K}.
\]
Maximizing the right-hand side over $|z|$ gives
\[
\norm{B_0(P_K)}_\infty
=(K+1)\left(\frac{2}{K+2}\right)^2
\left(\frac{K}{K+2}\right)^K
\le\frac{4}{K+1}. \qedhere
\]
\end{proof}

\subsection{Realization by integrable symbols}

The next lemma approximates $P_K$ by a Toeplitz operator in both trace norm and the localization seminorm needed above.

\begin{lemma}\label{lem:diagonal-smooth-realization}
For every $K\ge1$ and every $\delta>0$, there is a real radial function $h\in C_c^\infty(\mathbb D)$ such that
\[
T_h=P_K+E,
\]
where $E\ge0$ is trace class and
\[
\norm{E}_{\mathcal S_1}<\delta,
\qquad
\Lambda_3(E)\le C\norm{E}_{\mathcal S_1}.
\]
The constant $C$ is independent of $K$, $h$, and $E$.
\end{lemma}

\begin{proof}
We choose a nonnegative radial function $\psi\in C_c^\infty(\mathbb D)$ such that
\[
\int_{\mathbb D}\psi\,dA=1,
\]
and set $\psi_\epsilon(z)=\epsilon^{-2}\psi(z/\epsilon)$. We take $\epsilon>0$ small enough that $\psi_\epsilon$ is supported in $\mathbb D$, and define
\[
h_{K,\epsilon}
=\frac{\Delta^K\psi_\epsilon}
{(K+1)4^K(K!)^2},
\]
where $\Delta$ is the Laplacian. This function is real, radial, smooth, and compactly supported.

Since $h_{K,\epsilon}$ is radial, $T_{h_{K,\epsilon}}$ is diagonal in the basis $(e_j)_{j\ge0}$. We write
\[
T_{h_{K,\epsilon}}e_j=\lambda_j e_j.
\]
Integration by parts and the identity
\[
\Delta^K|z|^{2j}
=4^K\left(\frac{j!}{(j-K)!}\right)^2
|z|^{2(j-K)},
\qquad j\ge K,
\]
give
\[
\lambda_j=0\quad(j<K),
\qquad
\lambda_K=1,
\]
and, for $j>K$,
\begin{equation}\label{eq:diagonal-error-eigenvalues}
\lambda_j
=\frac{j+1}{K+1}\binom{j}{K}^2
\int_{\mathbb D}\psi_\epsilon(z)|z|^{2(j-K)}\,dA(z).
\end{equation}
In particular, $\lambda_j\ge0$.

If $\operatorname{supp}\psi\subset R\mathbb D$ for some $0<R<1$, then
\[
\lambda_{K+\ell}
\le
\frac{K+\ell+1}{K+1}\binom{K+\ell}{K}^2
(R\epsilon)^{2\ell},
\qquad \ell\ge1.
\]
For fixed $K$, the coefficient preceding $(R\epsilon)^{2\ell}$ grows polynomially in $\ell$. Hence
\[
\sum_{j>K}\lambda_j\longrightarrow0
\qquad\text{as }\epsilon\to0.
\]
Thus
\[
T_{h_{K,\epsilon}}=P_K+E_{K,\epsilon},
\qquad
E_{K,\epsilon}=\sum_{j>K}\lambda_jP_j,
\]
where $E_{K,\epsilon}\ge0$ and
\[
\norm{E_{K,\epsilon}}_{\mathcal S_1}
=\sum_{j>K}\lambda_j\longrightarrow0.
\]
Finally, Lemma~\ref{lem:diagonal-block-profile} and Minkowski's inequality imply
\[
\Lambda_3(E_{K,\epsilon})
\le\sum_{j>K}\lambda_j\Lambda_3(P_j)
\le C\norm{E_{K,\epsilon}}_{\mathcal S_1}.
\]
Choosing $\epsilon$ sufficiently small proves the lemma.
\end{proof}

\subsection{Proof of the endpoint theorem}

\begin{proof}[Proof of Theorem~\ref{thm:zorboska-endpoint}]
We choose strictly increasing integers $\ell_m\ge1$ such that
\begin{equation}\label{eq:degree-summability}
\sum_{m=1}^\infty\frac{1}{\ell_m+1}<\infty.
\end{equation}
By Lemma~\ref{lem:diagonal-smooth-realization}, choose real radial functions $h_m^0\in C_c^\infty(\mathbb D)$ such that
\[
T_{h_m^0}=P_{\ell_m}+E_m^0,
\qquad
E_m^0\ge0,
\qquad
\norm{E_m^0}_{\mathcal S_1}<2^{-m}.
\]

Set $\eta_m=2^{-m}$ and
\[
C_m=\left\{c\in\mathbb D:
L_{\ell_m}(|c|)\ge\eta_m\right\}.
\]
Lemma~\ref{lem:diagonal-block-profile} shows that $C_m$ is a compact subset of $\mathbb D$.

We now choose points $b_m\in(0,1)$ tending to $1$. Suppose that $b_1,\ldots,b_{m-1}$ have been chosen, and write
\[
w_j=U_{b_j}e_{\ell_j},
\qquad j<m.
\]
As $b\to1$, the following three facts hold:
\begin{enumerate}
\item $\varphi_b(C_m)$ converges uniformly to the boundary point $1$;
\item $U_be_{\ell_m}\rightharpoonup0$ in $A^2(\mathbb D)$;
\item
$\norm{h_m^0\circ\varphi_b}_{L^1(dA)}\to0$.
\end{enumerate}
The second fact was proved in Lemma~\ref{lem:weak-null}, and the third follows from Lemma~\ref{lem:decay}. We may therefore choose $b_m>1-2^{-m}$ so close to $1$ that
\begin{equation}\label{eq:active-sets-disjoint}
\varphi_{b_m}(C_m)\cap
\bigcup_{j<m}\varphi_{b_j}(C_j)=\varnothing,
\end{equation}
\begin{equation}\label{eq:diagonal-bessel-choice}
\sum_{j<m}
\abs{\inner{U_{b_m}e_{\ell_m}}{w_j}}^2
\le4^{-m},
\end{equation}
and
\begin{equation}\label{eq:diagonal-L1-choice}
\norm{h_m^0\circ\varphi_{b_m}}_{L^1(dA)}<2^{-m}.
\end{equation}
We set
\[
w_m=U_{b_m}e_{\ell_m},
\qquad
S_m=w_m\otimes w_m.
\]

For each $m$, \eqref{eq:diagonal-bessel-choice} and the corresponding estimates with indices greater than $m$ give
\[
\sum_{j\ne m}\abs{\inner{w_m}{w_j}}^2
\le\frac43\,4^{-m}.
\]
In fact, the Gram matrix of $(w_m)$ differs from the identity by a Hilbert--Schmidt matrix. In particular, $(w_m)$ is a Bessel sequence. Therefore
\[
Sg=\sum_{m=1}^\infty\inner{g}{w_m}w_m
\]
converges in $A^2(\mathbb D)$ for every $g$ and defines a bounded positive operator $S$. The sequence $(w_m)$ is weakly null, since it is Bessel, whereas
\[
\inner{Sw_m}{w_m}
=\sum_j\abs{\inner{w_m}{w_j}}^2\ge1.
\]
Thus $S$ is not compact.

We next verify the uniform localization estimate. By \eqref{eq:radial-localization-covariance},
\[
\norm{U_aS_mU_a1}_{L^3(dA)}
=L_{\ell_m}(|\varphi_{b_m}(a)|).
\]
In view of \eqref{eq:active-sets-disjoint}, for each $a\in\mathbb D$ there is at most one index $m$ for which this quantity is at least $\eta_m$. Hence Lemma~\ref{lem:diagonal-block-profile} gives
\begin{equation}\label{eq:S-localization}
\sum_m\norm{U_aS_mU_a1}_{L^3(dA)}
\le C+\sum_m\eta_m
\le C+1,
\end{equation}
uniformly in $a$. In particular, the series converges absolutely in $L^3(dA)$. Its partial sums converge in $A^2$ to $U_aSU_a1$; since $L^3(dA)$ embeds continuously into $L^2(dA)$, the $L^3$ sum has the same limit. Hence
\[
\Lambda_3(S)<\infty.
\]

We define
\[
f_m=h_m^0\circ\varphi_{b_m},
\qquad
f=\sum_{m=1}^\infty f_m.
\]
By \eqref{eq:diagonal-L1-choice}, $f\in L^1(\mathbb D,dA)$, and $f$ is real. Covariance gives
\[
T_{f_m}=S_m+F_m,
\qquad
F_m=U_{b_m}E_m^0U_{b_m}.
\]
The operator
\[
F=\sum_mF_m
\]
is positive and trace class, since
\[
\sum_m\norm{F_m}_{\mathcal S_1}
=\sum_m\norm{E_m^0}_{\mathcal S_1}<\infty.
\]
Moreover, \eqref{eq:radial-localization-covariance} and Lemma~\ref{lem:diagonal-smooth-realization} imply
\begin{equation}\label{eq:F-localization}
\Lambda_3(F)
\le\sum_m\Lambda_3(F_m)
=\sum_m\Lambda_3(E_m^0)
\le C\sum_m\norm{E_m^0}_{\mathcal S_1}<\infty.
\end{equation}

We claim that the Toeplitz form with symbol $f$ has the bounded extension
\[
T_f=S+F.
\]
Indeed, if $p,q$ are holomorphic polynomials, then the $L^1$ convergence of $\sum_mf_m$ gives
\[
\int_{\mathbb D}f(z)p(z)\overline{q(z)}\,dA(z)
=\sum_m\inner{T_{f_m}p}{q}.
\]
The series of matrix elements of $(S_m)$ converges absolutely by the Bessel property and Cauchy--Schwarz, while the series of $(F_m)$ converges in trace norm. The right-hand side is therefore $\inner{(S+F)p}{q}$.
This proves the claim. In particular, $T_f$ is bounded and positive. Moreover, $(w_m)$ is weakly null and
\[
\inner{T_fw_m}{w_m}
=\inner{Sw_m}{w_m}+\inner{Fw_m}{w_m}\ge1,
\]
so $T_f$ is not compact.

We now consider the Berezin transform. By covariance and \eqref{eq:diagonal-berezin-bound},
\[
\norm{B_0(S_m)}_\infty
=\norm{B_0(P_{\ell_m})}_\infty
\le\frac{4}{\ell_m+1}.
\]
Thus \eqref{eq:degree-summability} implies that $\sum_mB_0(S_m)$ converges uniformly. Each $S_m$ has rank one, so its Berezin transform vanishes at the boundary. It follows that $B_0(S)$ vanishes at the boundary. Since $F$ is compact, $B_0(F)$ also vanishes there. Therefore,
\[
B_0(f)=B_0(T_f)=B_0(S+F)\longrightarrow0
\qquad\text{as }|z|\to1.
\]

Finally, covariance, \eqref{eq:S-localization}, and \eqref{eq:F-localization} give
\[
\sup_{a\in\mathbb D}
\norm{T_{f\circ\varphi_a}1}_{L^3(dA)}
=\Lambda_3(T_f)<\infty.
\]
Because $f$ is real, $T_{\overline f\circ\varphi_a} = T_{f\circ\varphi_a},$ so the same $L^3$ estimate holds for the second term in \eqref{eq:zorboska-condition}. Since $dA$ is a probability measure, the two estimates remain valid in $L^p(dA)$ for every $2<p\leq3$.
\end{proof}
\section{Symbols outside the weighted Sobolev scale}

\begin{proposition}\label{prop:sobolev-obstruction}
Let $\nu(z)=1-|z|^2$. Suppose that
\[
g\in L^1(\mathbb D,dA)\cap W_\nu^{-m,\infty}(\mathbb D)
\]
for some $m\in\mathbb N_0$. Then the Toeplitz form with symbol $g$ has a bounded extension $T_g$ on $A^2(\mathbb D)$, and
\[
T_g\text{ is compact on }A^2(\mathbb D)
\quad\Longleftrightarrow\quad
B_0(g)(z)\longrightarrow0
\qquad\text{as }|z|\to1.
\]
Therefore, the symbol supplied by Theorem~\ref{thm:main} when $n=1$ and $\gamma=0$, and the symbol in Theorem~\ref{thm:zorboska-endpoint}, lie outside $W_\nu^{-m,\infty}(\mathbb D)$ for every finite $m$.
\end{proposition}

\begin{proof}
We combine the boundedness theorem of Per\"al\"a, Taskinen, and Virtanen with Zorboska's compactness theorem. We first verify that the weighted negative Sobolev norm remains uniformly controlled when the symbol is composed with a disk automorphism.

Recall that
\[
\norm{u}_{W_\nu^{m,1}}
=
\sum_{|\alpha|\leq m}
\int_{\mathbb D}
\abs{D^\alpha u(z)}\nu(z)^{|\alpha|}\,dA(z),
\]
and that $W_\nu^{-m,\infty}$ is isometrically the dual of $W_\nu^{m,1}$; see \cite[Definition~2.3 and Lemma~2.4]{PeralaTaskinenVirtanen2011}. For $a\in\mathbb D$, we put
\[
J_a=|\varphi_a'|^2,
\qquad
\mathcal C_a u=J_a(u\circ\varphi_a).
\]
Since $\varphi_a$ is an involution, a change of variables gives
\begin{equation}\label{eq:pullback-pairing}
\int_{\mathbb D}u(z)g(\varphi_a(z))\,dA(z)
=
\int_{\mathbb D}(\mathcal C_a u)(w)g(w)\,dA(w).
\end{equation}
Thus $\mathcal C_a$ is the operator on $W_\nu^{m,1}$ dual to the pullback $g\mapsto g\circ\varphi_a$, in agreement with the convention in \cite[Section~3.3]{PeralaTaskinenVirtanen2011Problems}.

We claim that
\begin{equation}\label{eq:weighted-sobolev-covariance}
\norm{\mathcal C_a u}_{W_\nu^{m,1}}
\leq C_m\norm{u}_{W_\nu^{m,1}},
\end{equation}
where $C_m$ is independent of $a$. Indeed,
\[
\nu(\varphi_a(z))=\nu(z)|\varphi_a'(z)|,
\qquad
J_a(z)=\frac{(1-|a|^2)^2}{|1-\overline a z|^4},
\]
and $\nu(z)\leq2|1-\overline a z|$. Hence, for every real multiindex $\alpha$ with $j=|\alpha|\geq1$,
\[
\nu(z)^j|D^\alpha\varphi_a(z)|
\leq C_j\nu(\varphi_a(z)),
\qquad
\nu(z)^j|D^\alpha J_a(z)|
\leq C_jJ_a(z),
\]
with $C_j$ independent of $a$. The chain and product rules give, for $|\alpha|\leq m$,
\[
\nu(z)^{|\alpha|}
\abs{D^\alpha(\mathcal C_a u)(z)}
\leq
C_mJ_a(z)
\sum_{|\beta|\leq|\alpha|}
\nu(\varphi_a(z))^{|\beta|}
\abs{D^\beta u(\varphi_a(z))}.
\]
Indeed, in each term containing $D^\beta u(\varphi_a)$, exactly $|\beta|$ derivatives of $\varphi_a$ occur, and their positive orders sum to the number of derivatives falling on $u\circ\varphi_a$. Integrating and using $J_a(z)dA(z)=dA(\varphi_a(z))$ proves \eqref{eq:weighted-sobolev-covariance}, first for $u\in C_c^\infty(\mathbb D)$ and then, by \cite[Lemma~2.2]{PeralaTaskinenVirtanen2011}, for every $u\in W_\nu^{m,1}$.

By the isometric duality in \cite[Lemma~2.4]{PeralaTaskinenVirtanen2011}, dualizing \eqref{eq:weighted-sobolev-covariance} through \eqref{eq:pullback-pairing} gives, with the same constant,
\begin{equation}\label{eq:negative-sobolev-covariance}
\sup_{a\in\mathbb D}
\norm{g\circ\varphi_a}_{W_\nu^{-m,\infty}}
\leq
C_m\norm{g}_{W_\nu^{-m,\infty}}.
\end{equation}

We next identify the operators furnished by \cite[Theorem~3.1]{PeralaTaskinenVirtanen2011} with the Toeplitz forms used here. Suppose that
\[
h\in L^1(\mathbb D,dA)\cap W_\nu^{-m,\infty}(\mathbb D).
\]
Let $u\in C^\infty(\overline{\mathbb D})$, and write $\langle u,h\rangle_\nu$ for the corresponding dual pairing. We choose $\delta_j\to 0$ and $\chi_j\in C_c^\infty(\mathbb D)$ such that $0\leq\chi_j\leq1$, $\chi_j=1$ on $\{|z|\leq1-2\delta_j\}$, $\chi_j=0$ on $\{|z|\geq1-\delta_j\}$, and
\[
\abs{D^\gamma\chi_j}\leq C_\gamma\delta_j^{-|\gamma|}
\qquad (|\gamma|\leq m).
\]
Then $\chi_j u\to u$ in $W_\nu^{m,1}$. Indeed, the derivatives of $\chi_j$ are supported where $\nu\sim\delta_j$, and the product rule gives
\[
\nu^{|\alpha|}
\abs{D^\gamma\chi_j}\abs{D^\beta u}
\leq
C_u\delta_j^{|\alpha|-|\gamma|}
=
C_u\delta_j^{|\beta|}
\]
whenever $\beta+\gamma=\alpha$; integration over the transition annulus supplies a further factor $O(\delta_j)$. The terms outside the transition annulus are handled directly by the shrinking support of $1-\chi_j$.

Moreover, for $v\in C_c^\infty(\mathbb D)$, the dual pairing $\langle v,h\rangle_\nu$ is the distributional pairing and hence equals $\int_{\mathbb D}hv\,dA$, since $h\in L^1(\mathbb D,dA)$. Therefore
\[
\langle u,h\rangle_\nu
=
\lim_{j\to\infty}\langle\chi_j u,h\rangle_\nu
=
\lim_{j\to\infty}\int_{\mathbb D}h\,\chi_j u\,dA
=
\int_{\mathbb D}hu\,dA,
\]
where the last equality follows by dominated convergence.

Let $S_h$ be the bounded operator on $A^2(\mathbb D)$ furnished by \cite[Theorem~3.1]{PeralaTaskinenVirtanen2011}. For a holomorphic polynomial $p$ and fixed $z\in\mathbb D$, apply the preceding identity to
\[
u(\zeta)=\frac{p(\zeta)}{(1-z\overline\zeta)^2}.
\]
Formula~(3.1) of that paper then gives
\[
(S_hp)(z)
=
\int_{\mathbb D}
\frac{h(\zeta)p(\zeta)}
     {(1-z\overline\zeta)^2}\,dA(\zeta).
\]
Comparing Taylor coefficients shows that, for holomorphic polynomials
$p,q$,
\[
\inner{S_hp}{q}
=
\int_{\mathbb D}h\,p\overline q\,dA.
\]
Thus $S_h$ is precisely the bounded extension $T_h$ of the Toeplitz form used here. Since formula~(3.1) is independent of the Bergman exponent, the same identification applies to the operators on $A^4$ used below. The same cutoff argument applied to $u=|k_z|^2$ also gives
\[
\widetilde h(z)
=
\langle |k_z|^2,h\rangle_\nu
=
\int_{\mathbb D}h\,|k_z|^2\,dA
=
B_0(h)(z),
\]
where $\widetilde h$ denotes the distributional Berezin transform used in \cite[Section~3.3]{PeralaTaskinenVirtanen2011Problems}.

We now apply the boundedness theorem. For every $h\in W_\nu^{-m,\infty}$ and every $1<p<\infty$, \cite[Theorem~3.1]{PeralaTaskinenVirtanen2011} gives
\[
\norm{T_h:A^p\to A^p}
\leq
C_{m,p}\norm{h}_{W_\nu^{-m,\infty}}.
\]
For each fixed $a\in\mathbb D$, the function $g\circ\varphi_a$ belongs to $L^1(\mathbb D,dA)$, so the preceding identification applies. Taking $p=4$ and $h=g\circ\varphi_a$, and using $\norm{1}_{A^4}=1$, we obtain
\begin{align*}
\norm{T_{g\circ\varphi_a}1}_{L^4(dA)}
&\leq
\norm{T_{g\circ\varphi_a}:A^4\to A^4}\norm{1}_{A^4}\\
&\leq
C_{m,4}\norm{g\circ\varphi_a}_{W_\nu^{-m,\infty}}\\
&\leq
C_{m,4}C_m\norm{g}_{W_\nu^{-m,\infty}}.
\end{align*}
The right-hand side is independent of $a$. Complex conjugation preserves $W_\nu^{-m,\infty}$, so the same argument applied to $\overline g$ gives
\[
\sup_{a\in\mathbb D}
\norm{T_{\overline g\circ\varphi_a}1}_{L^4(dA)}
<\infty.
\]
Thus \eqref{eq:zorboska-condition} holds with $p=4$.

Theorem~3.1 of \cite{PeralaTaskinenVirtanen2011}, applied with $p=2$, gives a bounded operator $S_g$ on $A^2(\mathbb D)$. The preceding identity for its polynomial matrix coefficients shows that $S_g$ is the bounded extension $T_g$ of the Toeplitz form with symbol $g$. This proves the asserted boundedness. If $B_0(g)$ vanishes at the boundary, then \cite[Theorem~4.2]{Zorboska2003} implies that $T_g$ is compact.

Conversely, suppose that $T_g$ is compact. By Lemma~\ref{lem:weak-null}, $k_z\rightharpoonup0$ as $|z|\to1$, and hence
\[
\norm{T_gk_z}_{A^2}\longrightarrow0.
\]
Since $k_z$ is bounded and holomorphic for each fixed $z$, Lemma~\ref{lem:l1-covariance} gives
\[
\abs{B_0(g)(z)}
=
\abs{\inner{T_gk_z}{k_z}}
\leq
\norm{T_gk_z}_{A^2}
\longrightarrow0.
\]
This proves the equivalence.

Finally, both disk symbols constructed here have boundary vanishing of their Berezin transforms and induce noncompact operators. If either symbol belonged to $W_\nu^{-m,\infty}$ for some finite $m$, then its Berezin transform which vanishes at the boundary, together with the implication just proved, would force its Toeplitz operator to be compact, a contradiction.
\end{proof}

\begin{remark}[Relation to Problems~1 and~8]Problem~1 of \cite{PeralaTaskinenVirtanen2011Problems} asks broadly for a characterization of bounded Toeplitz operators. Its final, more specific question asks whether membership in $W_\nu^{-m,\infty}$ for some finite $m$ is necessary for boundedness when the symbol lies in $L^1$. The disk examples give a negative answer to this specific question on $A^2(\mathbb D)$: their symbols lie in $L^1$ and induce bounded operators, but belong to none of the spaces $W_\nu^{-m,\infty}$.

Problem~8 of the same paper assumes $g\in W_\nu^{-m,\infty}$ and asks whether boundary vanishing of its distributional Berezin transform forces a representation
\[
g=\sum_{|\alpha|\leq m}(-1)^{|\alpha|}D^\alpha b_\alpha
\]
whose coefficients satisfy
\[
\lim_{r\to1} \operatorname*{ess\,sup}_{r<|z|<1} \nu(z)^{-|\alpha|}|b_\alpha(z)|=0.
\]
for every $|\alpha|\le m$. 
When $g$ also lies in $L^1$, the calculation above shows that its distributional Berezin transform is precisely $B_0(g)$: taking $u=|k_z|^2$ gives
\[
\widetilde g(z)
= \langle |k_z|^2,g\rangle_\nu
= \int_{\mathbb D}g\,|k_z|^2\,dA
= B_0(g)(z).
\]
Since compact operators have Berezin transforms that vanish at the boundary, the proposition yields, for every \[ g\in L^1(\mathbb D,dv_0)\cap W_\nu^{-m,\infty}(\mathbb D), \] the equivalence \[ T_g\text{ is compact on }A^2(\mathbb D) \quad\Longleftrightarrow\quad \widetilde g(z)=B_0(g)(z)\longrightarrow0 \quad\text{as }|z|\to 1. \]
This gives a necessary and sufficient compactness criterion on $A^2$ within the globally integrable subclass of $W_\nu^{-m,\infty}$, but does not produce the representation requested in Problem~8.  The symbols constructed here lie outside every such Sobolev space, so they do not answer that representation question.
\end{remark}

\section{Concluding remarks}\label{sec:concluding}
\begin{remark}[Relation to \cite{HeCao2013}]
We next check that the bounded extension of a Toeplitz form agrees with the densely defined Toeplitz operator used in \cite{HeCao2013}. Let $P$ denote the Bergman projection, and let $h\in L^1(\B_n,dv_\gamma)$ be such that its Toeplitz form has a bounded extension $T_h$. For $g\in H^\infty(\B_n)$ and $w\in\B_n$, the functions $g$ and $K_w$ are bounded and holomorphic, so Lemma~\ref{lem:l1-covariance} gives
\[
(T_hg)(w)=\inner{T_hg}{K_w}
=\int_{\B_n}\frac{h(z)g(z)}{(1-\inner{w}{z})^N}\dvgamma(z)
=P(hg)(w).
\]
Thus $T_h$ agrees on $H^\infty(\B_n)$ with the operator $g\mapsto P(hg)$ of \cite{HeCao2013}. Since $H^\infty(\B_n)$ contains the holomorphic polynomials, it is dense in $A^2_\gamma$, so the latter operator is bounded with bounded extension $T_h$.

Taking $h=f$, the symbol $f$ satisfies the hypotheses of \cite[Theorem~2.8]{HeCao2013} exactly as stated there: $f\in L^1(\B_n,dv_\gamma)$, $T_f$ is bounded on $A^2_\gamma$, and $B_\gamma(f)$ vanishes at the boundary. Since $T_f$ is not compact, and since \cite{HeCao2013} assumes $n\geq2$ throughout, this contradicts the conclusion of \cite[Theorem~2.8]{HeCao2013} in its full stated range.

The estimate above also shows that \cite[Lemma~2.7]{HeCao2013} fails at the right endpoint of its stated range. By Lemma~\ref{lem:l1-covariance}, the Toeplitz form with symbol $f\circ\varphi_a$ has the bounded extension $T_{f\circ\varphi_a}=U_aT_fU_a^*$, which, by the identification above, is the corresponding operator of \cite{HeCao2013}. Since $U_a$ is a self-adjoint unitary and $U_a1=k_a$, \[ \norm{T_{f\circ\varphi_a}1}_{A^2_\gamma} =\norm{U_aT_fk_a}_{A^2_\gamma} =\norm{T_fk_a}_{A^2_\gamma}, \] which is bounded away from zero along $(a_m)$. This contradicts the case $q=2$ of \cite[Lemma~2.7]{HeCao2013}. Our example does not refute the cases $1\le q<2$, which are the ones invoked in the proof of \cite[Theorem~2.8]{HeCao2013}.
\end{remark}

\begin{remark}[Historical note on Zorboska's question] Zorboska asked whether her compactness theorem remains valid when condition~\eqref{eq:zorboska-condition} is assumed for some $p>2$, rather than for some $p>3$.  A positive answer was announced in a 2007 preprint of Agbor \cite{Agbor2007}.  The author withdrew the preprint in 2012, and the result was never published. Theorem~\ref{thm:zorboska-endpoint} gives a negative answer for Toeplitz operators with $L^1$ symbols, already at the endpoint $p=3$. \end{remark}

\begin{remark}\label{rem:not-in-algebra} 
The operator $T_f$ of Theorem~\ref{thm:main} does not belong to the Toeplitz algebra $\mathcal T_\gamma$: otherwise the compactness theorem \cite{Suarez2007,MitkovskiSuarezWick2013} would force $T_f$ to be compact. More generally, Hagger proved that a bounded operator on $A^2_\gamma(\B_n)$ is compact if and only if it is band-dominated and its Berezin transform vanishes at the boundary \cite[Theorem~A]{Hagger2019}. Hence the operator constructed in Theorem~\ref{thm:main} is not band-dominated. 

In particular $f\notin L^\infty(\B_n,dv_\gamma)$, since bounded symbols give operators in $\mathcal T_\gamma$. Thus every example of this type must have an essentially unbounded symbol; our derivative-of-delta approximations provide one way to construct such a symbol.
\end{remark}
\section*{Acknowledgments} The author thanks Antti Per\"al\"a for helpful comments on an earlier version of the manuscript and, in particular, for drawing his attention to the work of Rozenblum and Vasilevski and to the literature on weak localization and the Toeplitz algebra.

\end{document}